\documentclass[12pt,oneside]{article}

\usepackage{setspace}
\usepackage[ruled]{algorithm2e}
\usepackage{bm,enumerate}
\usepackage{amssymb,amsfonts,amsthm}
\usepackage{amsmath}
\allowdisplaybreaks
\usepackage{mathtools,xfrac}
\usepackage{url,color}
\usepackage{graphicx}
\usepackage{xcolor}
\usepackage{subcaption}
\usepackage[utf8]{inputenc}
\usepackage[sort,numbers]{natbib}

\usepackage{verbatim}
\usepackage[colorlinks=true,breaklinks=true,bookmarks=true,urlcolor=blue,citecolor=blue,linkcolor=blue,bookmarksopen=false,draft=false]{hyperref}

\usepackage{tikz}
\usepackage{pgfplots}
\pgfplotsset{compat=1.3}

\usepackage[margin=1in]{geometry}
\makeatletter
\def\imod#1{\allowbreak\mkern10mu({\operator@font mod}\,\,#1)}
\makeatother
\usepackage[compact]{titlesec}

\theoremstyle{plain}
\newtheorem{theorem}{Theorem}[section]

\newtheorem{lemma}[theorem]{Lemma}
\newtheorem*{lemma*}{Lemma}

\newtheorem*{proposition*}{Proposition}
\newtheorem{corollary}[theorem]{Corollary}

\newtheorem*{observation*}{Observation}

\definecolor{gold}{rgb}{0.85,0.65,0}

\newcommand{\ip}[2]{\left\langle #1 , #2 \right\rangle}    %

\def\beq{\begin{equation}}
\def\eeq{\end{equation}}

\def\fnote#1{\footnote}

\def\R{{\mathbb{R}}}

\DeclareMathOperator{\Prox}{Prox}

\DeclareMathOperator{\argmin}{arg\,min}

\DeclareMathOperator{\dom}{dom}

\def\log{\mathop{{\rm log}}}

\begin{document}

\title{The Inexact Cyclic Block Proximal Gradient Method and Properties of Inexact Proximal Maps}
\author{Leandro Maia\thanks{Department of Industrial, Manufacturing and Systems Engineering,
Texas Tech University, USA, {\tt Leandro.Maia@ttu.edu}} \and David Huckleberry Gutman\thanks{Department of Industrial, Manufacturing and Systems Engineering,
Texas Tech University, USA, {\tt David.Gutman@ttu.edu}}
 \and Ryan Christopher Hughes\thanks{Addx Corporation, {\tt Ryan.Christopher.Hughes@gmail.com}}}

\maketitle

\begin{abstract}
    This paper expands the Cyclic Block Proximal Gradient method for block separable composite minimization by allowing for inexactly computed gradients and proximal maps. The resultant algorithm, the Inexact Cyclic Block Proximal Gradient (I-CBPG) method, shares the same convergence rate as its exactly computed analogue provided the allowable errors decrease sufficiently quickly or are pre-selected to be sufficiently small. We provide numerical experiments that showcase the practical computational advantage of I-CBPG for certain fixed tolerances of approximation error and for a dynamically decreasing error tolerance regime in particular. We establish a tight relationship between inexact proximal map evaluations and $\delta$-subgradients in our $\delta$-Second Prox Theorem. This theorem forms the foundation of our convergence analysis and enables us to show that inexact gradient computations and other notions of inexact proximal map computation can be subsumed within a single unifying framework. 
\end{abstract}

\section{Introduction}

We propose an Inexact Cyclic Block Proximal Gradient method (I-CBPG) for the block separable composite optimization problem
\begin{equation}\label{eq.problem}
F^*:=\min\left\{F(x):=f(x)+\sum_{i=1}^p \Psi_i\left(  U_i^T  x\right):x\in\R^n\right\}.
\end{equation}
We assume that $f:\R^n\to\R\cup\{\infty\}$ is smooth and convex, the matrices $U_i\in\R^{n\times n_i}$ are chosen such that $(U_1,\ldots,U_p)$ is an $n\times n$ permutation matrix, and each $\Psi_i: \R^{n_i} \to \R\cup\{\infty\}$ is proper, closed, and convex. Problem \eqref{eq.problem} naturally arises in data science whenever regularization is present. Matrix factorization \cite{Schmidt11}, LASSO \cite{Richtarik16}, group LASSO \cite{Qin13, Simon12}, matrix completion \cite{Wright09}, compressive sensing \cite{Donoho06, Wright09}, and neural network training \cite{Scardapane17} are but a few such problems.

The class of Block Proximal Gradient (BPG) methods readily exploits block separability to provide iterates that are cheap in terms of memory and computational costs, so they are popular for large-scale versions of problem \eqref{eq.problem} \cite{Beck13, FrongilloReid15, Leventhal10, Nesterov12,Richtarik16}. Often, BPG methods make considerable progress before a single full proximal gradient step can even completely execute. BPG methods principally differ in how they select the block $i$: greedily \cite{Richtarik12}, randomly \cite{Nesterov12}, or cyclically \cite{Beck13}. Cyclic BPG methods, the focus of our work, received their first convergence analysis in \cite{Beck13} which established the benchmark $\mathcal{O}(p/k)$ convergence rate when each $\Psi_i$ is the indicator of a closed and convex set. Later \cite{Shefi16} extended the analysis to account for $\Psi_i$ functions that are more generally proper, closed, and convex. Both \cite{Beck13} and \cite{Shefi16} assume exact computation of gradient and proximal maps and therefore avoid considering the effect of inexactness on their  analyses.

Since gradients and proximal maps are the main ingredients for a broad swath of first-order algorithms, the push to achieve lower iterate costs in large-scale settings has fueled research interest around their inexact computation (``inexactness''). Such inexactness provides a variety of benefits, but from a practical standpoint the most important is the ability to compute approximate updates quickly when a closed-form solution does not exist or would be prohibitively expensive from a computational perspective. The main focus of \cite{Schmidt11} is the convergence of the unaccelerated and accelerated proximal gradient schemes equipped with inexactly computed gradients and proximal maps. The ``inexact oracle" framework of \cite{Devolder14}, which is extended in \cite{Devolder13} and \cite{Dvurechensky16}, analyzes the convergence of common gradient based-methods when gradient or gradient-type mappings are computed inexactly. 

While some prior work has explored the effect of inexact computation on BPG methods, to our knowledge these studies have only concerned themselves with randomized schemes. A central inspiration for this work, \cite{Richtarik16}, considers how inexactly computed proximal maps and gradients affect the randomized BPG method. It further elaborates the benefits of incorporating pre-conditioning into prox map evaluations. Specifically, while pre-conditioning provides the benefit of making the problem of step-size selection trivial, this advantage comes at the cost of making closed-form evaluation of the pre-conditioned proximal map no longer possible in general. This lack of closed-form solution for the pre-conditioned proximal map drives the need for inexact proximal map evaluation. Although the randomized BPG is thoroughly studied in \cite{Richtarik16}, we are unaware of any works that attack the cyclic BPG variant. Our primary aim is to fill this apparent deficiency in the literature.

\subsection{Contributions and Outline}

We describe our key contributions along with the paper's layout below.
\begin{itemize}
	\item In Section \ref{section:inexact_proximal}, we analyze inexactly computed proximal maps that incorporate pre-conditioning in the sense of \cite{Richtarik16}. Our main theorem, the $\delta$-Second Prox Theorem (Theorem \ref{thm:second-prox}), generalizes what \cite{Beck17} calls the Second Prox Theorem \cite[Theorem 6.39]{Beck17} that supports the convergence proofs of a broad swath of proximal map-based algorithms. This Theorem's main feature is the tight relationship it expresses between inexact proximal map evaluations and $\delta$-subgradients of the underlying function.  This equivalence facilitates simple proofs of two important observations as corollaries. First, this paper's formulation of inexact proximal map computation subsumes another approach by Rockafellar in \cite{Rockafellar76} (Corollary \ref{cor:Rockafellar}). Second, instead of treating errors in proximal map and gradient computations separately it is feasible to regard them both more generally as inexactly computed proximal map evaluations (Corollary \ref{cor:gradient-inexact}).
	
	\item In Section \ref{section:inexact_block_method_composite}, we define and analyze our Inexact Cyclic Block Proximal Gradient (I-CBPG) method. To the best of our knowledge, this is the first coordinate descent-type scheme with deterministic guarantees that incorporates inexactly computed proximal maps and gradients. For $\ell_1$-norm regularized optimization problems, \cite{Hua12} analyzes a similar scheme. However, that analysis is intimately tied to properties of the $\ell_1$-norm. Our analysis provides two flavors of convergence results. First, we are able to show that, for a fixed tolerance of approximation error, the standard $\mathcal{O}(p/k)$ convergence rate for cyclic BPG method is preserved provided said error is pre-selected to be sufficiently small. Analogous results for randomized BPG methods with fixed errors are found in \cite{Richtarik16}. Second, we are able to show that said rate is preserved under the relatively loose condition that the error tolerance decreases at a $\mathcal{O}(1/k^2)$ rate. The decreasing error tolerance regime, in contrast to the fixed error tolerance regime, does not require any error tuning based on properties of the objective function, such as smoothness parameters, the initial optimality gap, or the initial iterate's distance from the set of optima. More importantly, as we see in our numerical experiments in Section \ref{sec:numerical}, the latitude that comes with looser approximations may yield significant speed advantages for early iterations in terms of CPU time. 
	\item In Section \ref{sec:numerical}, we provide numerical experiments on the well-known LASSO problem. These experiments demonstrate the power of our method and the particular benefits of dynamically decreasing error tolerance. 
\end{itemize}


%
%
\section{The Inexact Proximal Map and the $\delta$-Second Prox Theorem}
\label{section:inexact_proximal}

In this section, we introduce a framework for analyzing the effect of inexact computation on the \emph{pre-conditioned proximal map}
\begin{equation}\label{eq.prox:precond}
\Prox_{\Psi}^B(x,g):=\argmin_{y\in\R^n}\left\{\ip{g}{y}+\frac{1}{2}\|y-x\|_B^2+\Psi(y)\right\},
\end{equation}
where $x,g\in\R^n$, $\ip{\cdot}{\cdot}$ is an inner product on $\R^n$, $\|\cdot\|_B$ is the norm induced by the inner product $(x,y)\mapsto \ip{Bx}{y}$ with $B\in\R^{n\times n}$ positive definite, and $\Psi:\R^n\to\R\cup\{\infty\}$ is proper, closed, and convex. The dual norm of $\|\cdot\|_B$, which we denote $\|\cdot\|_B^*$, is easily shown to be $\|\cdot\|_{B^{-1}}$.

We must emphasize two crucial facts about the function \eqref{eq.prox:precond}. First, it is a generalization of the standard proximal map. Indeed, by setting $B=I_n$ and $g=0$ we recover 
\[
\Prox_{\Psi}^{I_n}(x,0)=\argmin_{y\in\R^n}\left\{\frac{1}{2}\|y-x\|^2+\Psi(y)\right\}=:\Prox_{\Psi}(x)
\]
Second, for common choices of $\Psi$ such as the $1$-norm, $\|\cdot\|_1$, the pre-conditioned proximal map does not have a closed-form expression unless $B$ is very simple, e.g. when $B=c\cdot I_n$ for some $c\in\R$. Outside of these special cases, one must usually recover $\Prox_{\Psi}^B(x,g)$ via numerical approximation. 

Instead of finding the unique, exact minimizer of $\Prox_{\Psi}^B$'s defining problem, though, our goal will be to find some $y\in\R^n$ that solves the problem up to a small predetermined error $\delta \in \R_+$. We are now prepared to formally define this section's centerpiece, the set-valued \emph{inexact pre-conditioned proximal map}, as the collection of all such approximate minima at $x$ with respect to $g,\delta,B,$ and $\Psi$:
\[
\Prox^B_\Psi(x,g,\delta):=\left\{y: \ip{g}{y}+\frac{1}{2}\|y-x\|_B^2+\Psi(y)\leq\min_{z\in\R^n}\left\{\ip{g}{z}+\frac{1}{2}\|z-x\|_B^2+\Psi(z)\right\}+\delta\right\}.
\]
One may also regard this as a generalized pre-conditioned proximal map, since we recover the exact pre-conditioned proximal map by setting $\delta=0$.

Our central result, the $\delta$-Second Prox Theorem (Theorem \ref{thm:second-prox}) is an inexact analogue of what \cite{Beck17} calls the  ``Second Prox Theorem", a key component in the bulk of convergence proofs for proximally-driven algorithms. Bregman-type generalizations of this theorem also support convergence proofs for Bregman proximal methods \cite{Lu17}. The $\delta$-Second Prox Theorem codifies the tight relationship between elements of $\Prox^B_\Psi$ and $\Psi$'s $\delta$-subdifferential. We say that $s\in\R^n$ is a $\delta$\emph{-subgradient of }$\Psi$\emph{ at }$x\in\dom(\Psi)$, with $\delta\geq 0$, if
\[
\Psi(y)\geq\Psi(x)+\ip{s}{y-x}-\delta\text{ for all }y\in\R^n.
\]
The $\delta$\emph{-subdifferential of }$\Psi$\emph{ at }$x$, $\partial \Psi_\delta(x)$, denotes the set of all $\delta$-subgradients of $\Psi$ at $x$. Rudiments of the relationship between $\Prox^B_\Psi$ and $\partial_\delta\Psi(x)$ appear in \cite{Schmidt11}, where it is shown that each $u\in\Prox^B_\Psi(x,g,\delta)$ associates to certain norm-bounded $s\in\partial_{\delta'}\Psi(x)$ for judicious selections of $\delta'$. Our $\delta$-Second Prox Theorem secures this result and its converse: each $\delta$-subgradient of $\Psi$ corresponds to a specially chosen inexact proximal map element. 

This equivalent characterization is particularly useful for verifying if $u\in\Prox_\Psi(x,g,\delta)$ as the reader will see in the last half of this section. Indeed, verification of this condition yields two important facts about the inexact pre-conditioned proximal map. First, it allows us to show that our approach to proximal map approximation includes another one defined in \cite{Rockafellar76} as a special case (Corollary \ref{cor:Rockafellar}). Second, it shows that we can regard approximation errors in gradient and proximal map computations simply as inexact proximal map evaluations (Corollary \ref{cor:gradient-inexact}).

Before stating and proving the $\delta$-Second Prox theorem, we recall two rules of the $\delta$-subdifferential calculus that are crucial to its proof, a sum rule and an optimality condition.

\begin{theorem}[$\delta$-Subdifferential Calculus]\label{thm:optimality:approx}
If $\Psi,\Psi':\R^n\to\R\cup\{\infty\}$ are proper, closed, and convex, and $\delta\geq 0$ then
\begin{enumerate}[(i)]
	\item (Optimality Condition)  It holds that $\Psi(x)-min_{y\in\R^n}\Psi(y)\leq\delta$ if and only if $0\in\partial_\delta\Psi(x)$ \cite[Theorem XI 1.1.5]{Hiriart13}.
	\item (Sum Rule) If $\text{ri}[\dom(\Psi)]\cap\text{ri}[\dom(\Psi')] \neq \emptyset$, where $\text{ri}(\cdot)$ denotes the relative interior of a convex set, then 
	\[
	\partial_\delta [\Psi+\Psi'](x)=\bigcup_{\delta'\in[0,\delta]}\left[\partial_{\delta'}\Psi(x)+\partial_{(\delta-\delta')}\Psi'(x)\right]
	\]
	for all $x\in\dom(\Psi)\cap\dom(\Psi')$, where the sum on the right is taken in the Minkowski sense \cite[Theorem XI 3.1.1]{Hiriart13}.
\end{enumerate}
\end{theorem}

We now state and prove the primary result of this section, the $\delta$-Second Prox Theorem. All other results in this section, along with the majority of those in its sequel, hang on this theorem.

\begin{theorem}[$\delta$-Second Prox]\label{thm:second-prox}
Let $\Psi:\R^n\to\R\cup\{\infty\}$ be proper, closed, and convex, $\delta\geq 0$, and $B\succ 0$. The following are equivalent:
\begin{enumerate}[(i)]
	\item $u\in \Prox^B_\Psi(x,g,\delta)$
	\item There exists $\delta'\in[0,\delta]$ and $v\in\R^n$ such that $\|v\|_B^*\leq\sqrt{2(\delta-\delta')}$ and $v-g-B(u-x)\in\partial_{\delta'}\Psi (u)$.
	\item There exists $\delta'\in[0,\delta]$ and $v\in \R^n$ such that $\|v\|_B^*\leq\sqrt{2(\delta-\delta')}$ and 
\begin{equation}\label{eqn:approxsecond}
\ip{v-g-B(u-x)}{y-u}\leq \Psi(y)-\Psi(u)+\delta'
\end{equation}
for all $y\in\R^n$.
\end{enumerate}
\end{theorem}

\begin{proof}
The equivalence of (ii) and (iii) is immediate from the definition of $\partial_{\delta'}\Psi(u)$ so it suffices to show the equivalence of (i) and (ii). The crux of this equivalence's proof is the expression of the $\delta$-subgradient for the function $z\mapsto \ip{g}{z}+\frac{1}{2}\|z-x\|_B^2+\Psi(z)$ for $x$ fixed,
\begin{equation}\label{eqn:approx:subgrad:sum}
\partial_{\delta}\left[\ip{g}{\cdot}+\frac{1}{2}\|\cdot-x\|_B^2+\Psi(\cdot)\right](z)=g+\bigcup_{\delta'\in[0,\delta]}\left[\partial_{(\delta-\delta')} \frac{1}{2} \|\cdot-x\|_B^2(z)+\partial_{\delta'}\Psi(z)\right],
\end{equation}
which is a consequence of Theorem \ref{thm:optimality:approx}(ii). A straightforward computation produces
\[
\partial_{(\delta-\delta')} \frac{1}{2} \|\cdot-x\|_B^2(z)=B(z-x)+\left\{v:\|v\|_B^*\leq\sqrt{2(\delta-\delta')}\right\}
\]
where the sum is taken in the Minkowski sense. Thus, we write \eqref{eqn:approx:subgrad:sum} more explicitly as
\begin{multline*}
\partial_{\delta}\left[\ip{g}{\cdot}+\frac{1}{2}\|\cdot-x\|_B^2+\Psi(z)\right](z)=\\
g+B(z-x)+\bigcup_{\delta'\in[0,\delta]}\left[\partial_{\delta'}\Psi(z)+\left\{v:\|v\|_B^*\leq\sqrt{2(\delta-\delta')}\right\}\right]
\end{multline*}
In light of the $\delta$-subdifferential optimality condition (Theorem \ref{thm:optimality:approx}(i)), $u\in\Prox^B_\Psi(x,g,\delta)$ if and only if
\[
0\in g+B(u-x)+\bigcup_{\delta'\in[0,\delta]}\left[\partial_{\delta'}\Psi(u)+\left\{v:\|v\|_B^*\leq\sqrt{2(\delta-\delta')}\right\}\right].
\]
Noting that $\|v\|_B^* = \|-v\|_B^*$ and rearranging, this inclusion is clearly equivalent to 2).
\end{proof}

A notable product of the $\delta$-Second Prox theorem is that inexact proximal maps exhibit Lipschitz continuity in $x,g$ up to the level of inexactness, as the theorem below formalizes. It will play the same role in the convergence proof for I-CBPG as its exact analogue does in the convergence proof of the Cyclic Block Proximal Gradient method (compare Lemmata \ref{lemma:prox:block-decrease} and \ref{lemma:prox:convex:sufficient-decrease} with \cite[Lemma 11.11]{Beck17} and \cite[Lemma 11.16]{Beck17}).
\newpage
\begin{theorem}[Error Dependent Lipschitz Continuity of the $\delta$-Prox Map]\label{thm:non-expansive}
Let $\Psi:\R^n\to\R\cup\{\infty\}$ be proper, closed, and convex, $x,y,g,h\in\R^n$  $\delta,\epsilon\geq 0$, and $B\succ 0$.Then
\begin{equation}\label{eq:non-expansive}
\|u-w\|_B\leq\|g-h\|_B^*+\|y-x\|_B+\left(1+\frac{\sqrt{2}}{2}\right)\cdot\left(\sqrt{\delta}+\sqrt{\epsilon}\right)
\end{equation}
for all $u\in \Prox^B_\Psi(x,g,\delta)$ and $w\in \Prox^B_\Psi(y,h,\epsilon)$.
\end{theorem}

\begin{proof}
Invoking the $\delta$-Second Prox theorem, we may choose $v_u,v_w\in \R^n$, $\delta'\in[0,\delta]$, and $\epsilon'\in[0,\epsilon]$, such that $\|v_u\|_B^*\leq\sqrt{2(\delta-\delta')}$, $\|v_w\|_B^*\leq\sqrt{2(\epsilon-\epsilon')}$, and 
\begin{align*}
\ip{v_u-g-B(u-x)}{z-u}&\leq \Psi(z)-\Psi(u)+\delta'\\
\ip{v_w-h-B(w-y)}{z'-w}&\leq \Psi(z')-\Psi(w)+\epsilon'
\end{align*}
for all $z,z'\in\R^n$. If we add these two inequalities with $z=w$ and $z'=u$ then
\[
\ip{v_u-g-B(u-x)}{w-u}+\ip{v_w-h-B(w-y)}{u-w}\leq\delta'+\epsilon',
\]
which simplifies to the more informative
\[
\|u-w\|_B^2+\ip{(g-h)+(v_w-v_u)+B(y-x)}{u-w}-(\delta'+\epsilon')\leq 0.
\]
The standard Cauchy-Schwarz inequality, in conjunction with its more general form for arbitrary norms and their dual norms along with the triangle inequality, implies
\[
\|u-w\|_B^2-\left(\|(g-h)+(v_w-v_u)\|_B^*+\|y-x\|_B\right)\cdot\|u-w\|_B-(\delta'+\epsilon')\leq 0.
\]
The left-hand side of this inequality is quadratic in $\|u-w\|_B$. Thus we have the generic bound
\begin{multline}\label{eq:non-expansiveness-quadratic}
\|u-w\|_B\leq\frac{\|(g-h)+(v_w-v_u)\|_B^*+\|y-x\|_B}{2}\\+\frac{\sqrt{(\|(g-h)+(v_w-v_u)\|_B^*+\|y-x\|_B)^2+4(\delta'+\epsilon')}}{2}
\end{multline}

We now reason by cases to derive four specialized versions of \eqref{eq:non-expansiveness-quadratic} (equations \eqref{eq:non-expansive-1}, \eqref{eq:non-expansive-2},  \eqref{eq:non-expansive-3}, and \eqref{eq:non-expansive-4})  that we chain together via the triangle inequality in \eqref{eq:non-expansive-triangle} to secure our result \eqref{eq:non-expansive}. First, with $x$, $g$, $\delta$, and $u$  fixed as above, suppose that $y = x$, $h = g$, and $\epsilon = 0$. Then $w=\Prox_\Psi^B(x,g,0)$, $v_w=0$, and the generic bound  \eqref{eq:non-expansiveness-quadratic} reduces to  
\begin{align}
\left\|u-\Prox_\Psi^B(x,g,0)\right\|_B\leq\frac{\|v_u\|_B^*+\sqrt{\left(\|v_u\|_B^*\right)^2+4\delta'}}{2}&\leq\frac{\sqrt{2\delta}+\sqrt{2(\delta+\delta')}}{2}\notag\\
&\leq\left(1+\frac{\sqrt{2}}{2}\right)\cdot\sqrt{\delta}\label{eq:non-expansive-1}
\end{align}
with the latter two inequalities respectively resulting from $\|v_u\|_B^*<\sqrt{2(\delta-\delta')}$ and $0\leq\delta'\leq\delta$. Second, with $x,g,$ and $h$ similarly fixed, let $y = x$ and $\delta=\epsilon=0$. Obviously,  $v_u=v_w=0$, $u=\Prox_\Psi^B(x,g,0)$, and $w=\Prox_\Psi^B(x,h,0)$ so now \eqref{eq:non-expansiveness-quadratic} reduces to
\begin{equation}\label{eq:non-expansive-2}
\|\Prox_\Psi^B(x,g,0)-\Prox_\Psi^B(x,h,0)\|_B\leq\frac{\|g-h\|_B^*}{2}+\frac{\sqrt{(\|g-h\|_B^*)^2}}{2}=\|g-h\|_B^*.
\end{equation}
Third, with $x$, $y$, and $h$ fixed as above, suppose that $\delta=\epsilon=0$ and $g=h$. Then we obtain
\begin{equation}\label{eq:non-expansive-3}
\|\Prox_\Psi^B(x,h,0)-\Prox_\Psi^B(y,h,0)\|_B\leq\frac{\|y-x\|_B}{2}+\frac{\sqrt{(\|y-x\|_B)^2}}{2}=\|y-x\|_B
\end{equation}
since here $v_u=v_w=0$, $u=\Prox_\Psi^B(x,h,0)$, and $w=\Prox_\Psi^B(y,h,0)$. Finally, with $y$, $h$, and $\epsilon$ fixed at their original values, we may let $x = y$, $g = h$, and $\delta = 0$ to see that, by the same logic as  \eqref{eq:non-expansive-1},
\begin{equation}\label{eq:non-expansive-4}
\|\Prox_\Psi^B(y,h,0)-w\|_B \leq 
\frac{\|v_w\|_B^*+\sqrt{\left(\|v_w\|_B^*\right)^2+4\epsilon'}}{2}\leq \left(1+\frac{\sqrt{2}}{2}\right)\cdot\sqrt{\epsilon}
\end{equation}

Now, we may complete our proof. Furnished with \eqref{eq:non-expansive-1}, \eqref{eq:non-expansive-2}, \eqref{eq:non-expansive-3}, and  \eqref{eq:non-expansive-4}, we compute
\begin{multline}\label{eq:non-expansive-triangle}
\|u-w\|_B\leq\|u-\Prox_\Psi^B(x,g,0)\|_B\\+
\|\Prox_\Psi^B(x,g,0)-\Prox_\Psi^B(x,h,0)\|_B+\|\Prox_\Psi^B(x,h,0)-\Prox_\Psi^B(y,h,0)\|_B\\
+\|\Prox_\Psi^B(y,h,0)-w\|_B\\
\leq \left(1+\frac{\sqrt{2}}{2}\right)\cdot\sqrt{\delta}+\|g-h\|_B^*+\|y-x\|_B+\left(1+\frac{\sqrt{2}}{2}\right)\cdot\sqrt{\epsilon},
\end{multline}
for all $u\in \Prox^B_\Psi(x,g,\delta)$ and $w\in \Prox^B_\Psi(y,h,\epsilon)$, which is precisely \eqref{eq:non-expansive}.
\end{proof}

Let us now discuss how the inexact pre-conditioned proximal map includes other approaches to proximal map and gradient approximation. We show our framework subsumes at least two others. First, we turn to one proposed in \cite{Rockafellar76} by Rockafellar. Rockafellar approximates the (exact) proximal mapping $x\mapsto\Prox_\Psi(x):=\Prox^{I_n}_\Psi(x,0,0)$ by way of $y\in\R^n$ satisfying $\|r+y-x\|\leq\delta$ for some $r\in\partial\Psi(y)$ and $\delta >0$. The corollary below explains how Rockafellar's approximation method is a special case of our own.

\begin{corollary}\label{cor:Rockafellar}
Let $\Psi:\R^n\to\R\cup\{\infty\}$ be proper, closed, and convex, $\delta\geq 0$, and $B\succ 0$. For all $x,g\in\R^n$, $\Prox^B_\Psi\left(x,g,\delta\right)$ is equal to the set of $u\in\R^n$ satisfying
\[
\left\|r+g+B(u-x)\right\|_B^*\leq\sqrt{2(\delta-\delta')}
\]
for some $\delta'\in[0,\delta]$ and $r\in\partial_{\delta'}\Psi(u)$. In particular, for the set of approximate proximal map evaluations in the sense of \cite{Rockafellar76},
\[
\left\{u:\exists r\in\partial\Psi(u)\text{ s.t. }\|u+r-x\|\leq\sqrt{2\delta}\right\}\subseteq\Prox^{I_n}_\Psi\left(x,0,\delta\right)
\]
holds where $I_n$ denotes the $n\times n$ identity matrix.
\end{corollary}

\begin{proof}
The first part of the corollary directly follows from Theorem \ref{thm:second-prox}(iii) with $v=r+g+B(u-x)$. The second part follows immediately from the first.
\end{proof}

In the context of first-order, proximally-based algorithms, a number of researchers study approximation error in just one of either the proximal map or gradient \cite{Alexandre_2008,Villa13,Devolder14}, or, in cases where both types of errors are considered, their treatment is often handled separately \cite{Schmidt11}. For example, \cite{Schmidt11}, one of the authoritative works on the proximal gradient scheme with errors for the composite minimization problem $\min_{x\in\R^n} f(x)+\Psi(x)$ tenders the error dependent scheme
\begin{align*}
	x_k=&\Prox_\Psi^{I_n}\left(y_{k-1},t\left(\nabla f(y_{k-1})+e_k\right),\delta_k\right)\\
	y_k&=x_k+\beta_k(x_k-x_{k-1})
\end{align*}
where $\{e_k\}_{k\geq 1}$ records the gradient approximation error, $\{\beta_k\}\subseteq[0,\infty)$ dictates the momentum used to accelerate the algorithm, and $t>0$ is a stepsize parameter. 

Theorem \ref{thm:non-expansive} unveils an intriguing property of our $\delta$-Second Prox Theorem framework for analyzing the inexact proximal map: it is possible to unify the treatment of inexactly computed proximal maps and gradients by simply considering inexact proximal map computations. This observation is formalized via the next corollary's set inclusion.

\begin{corollary}\label{cor:gradient-inexact}
Let $\Psi:\R^n\to\R\cup\{\infty\}$ be proper, closed, and convex, and $B\succ 0$. For all $x,g,e\in\R^n$ and $\delta\geq 0$, we have the inclusion
\[
\Prox_\Psi^B(x,g+e,\delta)\subseteq\Prox_\Psi^B\left(x,g,\delta+\sqrt{2\delta}\|e\|_B^*+\frac{1}{2}\left(\|e\|_B^*\right)^2\right).
\]
\end{corollary}

\begin{proof}
This follows from Theorem \ref{thm:second-prox}(ii). Fix $u\in\Prox_\Psi^B(x,g+e,\delta)$. There exist $\delta'\in[0,\delta]$ and $v\in\R^n$ such that $\|v\|_B^*\leq\sqrt{2(\delta-\delta')}$ and $v-(g+e)-B(u-x)\in\partial_{\delta'}\Psi(u)$. Equivalently, $(v-e)-g-B(u-x)\in\partial_{\delta'}\Psi(u)$. Thus, $u\in\Prox_\Psi^B\left(x,g,\delta+\sqrt{2\delta}\|e\|_B^*+\frac{1}{2}\left(\|e\|_B^*\right)^2\right)$ since
\[
\|v-e\|_B^*\leq\sqrt{2\left[\left(\frac{1}{2}\left(\|v-e\|_B^*\right)^2+\delta'\right)-\delta'\right]}\leq\sqrt{2\left[\left(\delta+\sqrt{2\delta}\|e\|_B^*+\frac{1}{2}\left(\|e\|_B^*\right)^2\right)-\delta'\right]}.
\]
\end{proof}

%
%
\section{The Inexact Cyclic Block Proximal Gradient Method}
\label{section:inexact_block_method_composite}

In this section we introduce and analyze the Inexact Cyclic Block Proximal Gradient method (I-CBPG), a variant of the cyclic block proximal gradient method for \eqref{eq.problem} that allows for approximate evaluations of (pre-conditioned) proximal maps and gradients. Throughout, we will assume that the smooth and convex component $f$ of \eqref{eq.problem} satisfies the following \emph{block smoothness} condition: for each $i=1,\ldots,p$ there exists positive definite matrix $B_i\in\R^{n_i\times n_i}$ and $L_i>0$ such that
\begin{equation}\label{eq:smooth:block}
f(x+U_i t)\leq f(x)+\ip{\nabla_i f(x)}{t}+\frac{L_i}{2}\|t\|_{(i)}^2\text{ for all }x\in\R^n\text{ and }t\in\R^{n_i},
\end{equation}
where $\|\cdot\|_{(i)}$ denotes the norm on $\R^{n_i}$ induced by the inner product $(x,y)\mapsto\ip{B_i x}{y}$ and $\nabla_i f(x) = U_i^T \nabla f(x)$. We will also assume that $f$ satisfies a (pre-conditioned) \emph{smoothness} condition: there exists $L_f>0$ such that\begin{equation}\label{eq:smooth}
f(x+t)\leq f(x)+\ip{\nabla f(x)}{t}+\frac{L_f}{2}\|t\|_{B}^2\text{ for all }x,t\in\R^n,
\end{equation}
where $\|\cdot\|_B=\sqrt{\sum_{i=1}^p\|U_i^T \cdot\|_{(i)}^2}$. Finally, we will assume that $F$ is coercive, i.e. it has bounded sublevel sets. In particular, this implies that $R(x):=\sup_{y\in X^*}\|x-y\|<\infty$ where $X^* := \argmin_x f(x)$.

To define the algorithm, we assume for each $i=1,\ldots,p$, the set-valued map $\Prox^{B_i}_{\Psi_i/L_i}$ admits a selection function
\[
(x,\delta)\mapsto T_\delta^{(i)}(x)\in\Prox^{B_i}_{\frac{\Psi_i}{L_i}}\left(U_i^Tx,\frac{1}{L_i}\nabla_i f(x),\delta\right) \subseteq \R^{n_i} 
\] 
that ensures the monotonic decrease condition
\begin{equation}\label{eq.mondec}
f\left( x+U_i[T_\delta^{(i)}(x)-x_i] \right)\leq f(x)
\end{equation}
where $x_i \in \R^{n_i}$ is the $i$th block of $x$, that is, $x_i = U_i^T x$. It is now possible to introduce our I-CBPG scheme.\\

\begin{algorithm}[H]
\caption{Inexact Cyclic Block Proximal Gradient (I-CBPG) Method}\label{algo.main}
	\KwData{$x^0 \in \dom(f)$, $\delta_1\geq 0$}
	\For{k=0,1,2,\ldots}{
		$x^{k,0}=x^k$\;
		\For{i=1,\ldots,p}{
		\begin{equation}
			x^{k,i}=x^{k,i-1}+U_i[T_{\delta_{k+1}}^{(i)}(x^{k,i-1})-x^{k,i-1}_i];
		\end{equation}
		}
		$x^{k+1}:= x^{k,p}$\; 
		Choose $\delta_{k+2}\in[0,\delta_{k+1}]$\;		
		}	
\end{algorithm}
\noindent We call $\{\delta_k\}_{k\geq 1}$ the sequence of \emph{error tolerances}. Two types of error tolerance sequences will be considered: \emph{fixed} sequences where merely assume that $\delta_k=\delta\geq 0$ for $k\geq 1$ and \emph{dynamically decreasing} sequences that converge to $0$ at the sublinear rate $\mathcal{O}(1/k^2)$.
 
Our analysis of I-CBPG (Algorithm \ref{algo.main}) follows a standard three step outline for proving convergence of a first-order method purposed for convex minimization. First, we prove a sufficient decrease condition (Lemma \ref{lemma:prox:block-decrease}) that relates the suboptimality gap to the norm of the inter-iterate difference, $x^{k,i}-x^{k,i-1}$. Second, using the sufficient decrease condition, we derive a recurrence inequality (Lemma \ref{lemma:prox:convex:sufficient-decrease}) satisfied by the sequence of suboptimality gaps at each iterate. Third, we prove a technical lemma (Lemma \ref{lemma:recurrence-technical}) that describes the rate of convergence of a recurrence of the form found in Lemma \ref{lemma:prox:convex:sufficient-decrease}. Finally, we deduce our desired convergence rates as a consequence of the technical lemma and the suboptimality gap recurrence inequality. The convergence rates for fixed errors are summarized in Theorem \ref{theorem:convergence-bounded-error} and Corollary \ref{cor:convergence-bounded-error} while the convergence rates for sublinearly decreasing errors are summarized in Theorem \ref{theorem:convergence-decreasing-error} and Corollary \ref{cor:convergence-decreasing-error}.

We begin by presenting the sufficient decrease inequality.

\begin{lemma}[Sufficient Decrease Inequalities]\label{lemma:prox:block-decrease}
Let $\{x^k\}_{k\geq 0}$ and $\{\delta_k\}_{k\geq 1}$ denote the sequences of iterates and error tolerances generated by I-CBPG (Algorithm \ref{algo.main}). Then we have that 
\begin{enumerate}[(i)]
	\item For all $k\geq 0$ and $1\leq i\leq p$,
\begin{equation}\label{eqn:prox:block-decrease:main-inexact}
3L_i\delta_{k+1}+F(x^{k,i})-F(x^{k,i-1})\geq\frac{L_i}{4}\left\|x^{k,i}-x^{k,i-1}\right\|_{(i)}^2
\end{equation}
	\item For all $k\geq 0$,
	\begin{equation}
3L_{\min}p\delta_{k+1}+F(x^k)-F(x^{k+1})\geq\frac{L_{\min}}{4}\| x^k - x^{k+1}	\|_B^2
\end{equation}
\end{enumerate}
\end{lemma}

\begin{proof} 
To streamline notation for the proofs of (i) and (ii), let us define the variables $x=x^{k,i-1}$,  $x^+=x^{k,i}$, and $\delta = \delta_{k+1}$. First, notice that the $i$-th block of $x^+$ is $x^+_i=T_\delta^{(i)}(x)$.\\

\noindent (i) By the blockwise smoothness property,
\[
f(x^+)\leq f(x)+\ip{\nabla_i f(x)}{x^+_i-x_i}+\frac{L_i}{2}\|x^+_i-x_i\|_{(i)}^2
\]
so
\begin{align}
F(x^+)&\leq F(x)+\ip{\nabla_i f(x)}{x^+_i-x_i}+\frac{L_i}{2}\left\|x^+_i-x_i\right\|_{(i)}^2+\Psi_i(x^+_i)-\Psi_i(x_i)\label{eqn:prox:block-decrease:smooth1}
\end{align}
We aim to refine the right-hand side of this inequality. To this end, the Theorem \ref{thm:second-prox}(iii) gives us 
\[
\ip{v-\frac{1}{L_i}\nabla_i f(x)-B_i(x^+_i-x_i)}{x_i-x^+_i}\leq \frac{\Psi_i}{L_i}(x_i)-\frac{\Psi_i}{L_i}\left(x^+_i\right)+\delta,
\]
for some $v\in\R^n$ such that $\|v\|_{(i)}^*\leq \sqrt{2\delta}$, since $x^+_i=T_\delta^{(i)}(x)$. We rearrange this inequality to bound $\ip{\nabla_i f(x)}{x^+_i-x_i}$ according to
%
%
\begin{align}
\ip{\nabla_i f(x)}{x^+_i-x_i}&\leq-L_i\|x^+_i-x_i\|_{(i)}^2+\Psi_i(x_i)-\Psi_i\left(x^+_i\right)+L_i\ip{v}{x^+_i-x_i}+L_i\delta\notag\\
& \leq  - L_i\|x^+_i-x_i\|_{(i)}^2+\Psi_i(x_i)-\Psi_i\left(x^+_i\right)\notag\\
&\hspace{10em}+\frac{L_i}{2}\Big( 2(\|v\|_{(i)}^*)^2 + \frac{1}{2}\|x^+_i-x_i\|_{(i)}^2 \Big)+L_i\delta\label{eq:CS-app-1}\\
& \leq  - \frac{3L_i}{4}\|x^+_i-x_i\|_{(i)}^2+\Psi_i(x_i)-\Psi_i\left(x^+_i\right)+3L_i\delta\label{eq:second-prox-grad-bound}
\end{align}
where on the third line we applied the AM-GM inequality to $\ip{v}{x^+_i-x_i}=\ip{\sqrt{2}v}{\frac{1}{\sqrt{2}}(x^+_i-x_i)}$. Finally, inserting \eqref{eq:second-prox-grad-bound}'s bound into \eqref{eqn:prox:block-decrease:smooth1}'s right-hand side, we settle on a rearranged \eqref{eqn:prox:block-decrease:main-inexact},
\[
F(x^+)\leq F(x)-\frac{L_i}{4}\left\|x^+_i-x_i\right\|_{(i)}^2+3L_i\delta.
\]
\noindent (ii) Rearrange the below chain of inequalities that follows from applying (i)
\begin{align*}
\| x^k - x^{k+1}	\|_B^2=\sum_{i=1}^{p}\|x^{k,i}-x^{k,i-1}\|_{(i)}^2&\leq\sum_{i=1}^{p}\left\{\frac{4}{L_i}[F(x^{k,i-1})-F(x^{k,i})]+12\delta\right\}\\
&\leq\frac{4}{L_{\min}}[F(x^k)-F(x^{k+1})]+12p\delta.
\end{align*}
\end{proof}
We now take the second step in our analysis: deriving the main recurrence inequality.

\newpage
\begin{lemma}\label{lemma:prox:convex:sufficient-decrease}
Let $\{x^k\}_{k\geq 0}$ and $\{\delta_k\}_{k\geq 1}$ denote the sequences of iterates and error tolerances generated by I-CBPG (Algorithm \ref{algo.main}). Then for all $k\geq 0$ the recurrence inequality
\[
\frac{L_{\min}}{2p(L_f+L_{\max})^2 R(x^0)^2}[F(x^{k+1})-F^*]^2\leq F(x^k)-F(x^{k+1})+ \mathcal{O}(\delta_{k+1})
\]
holds. Specifically,
\begin{multline}\label{eq:recur-main}
\frac{L_{\min}}{8p(L_f+L_{\max})^2 R(x^0)^2}[F(x^{k+1})-F^*]^2
\leq\\ F(x^k)-F(x^{k+1})+ L_{\min}\left[3p +    \frac{L^2_{\max}}{4}\left(  \frac{R(x^0)\sqrt{2} + \sqrt{p\delta_1}}{(L_f + L_{\max})R(x^0)} \right)^2\right]\delta_{k+1}
\end{multline}
\end{lemma}

\begin{proof}
Fix $k\geq 0$ and $i\in\{1,\ldots,p\}$. Invoking the $\delta$-Second Prox Theorem (Theorem \ref{thm:second-prox}) for $T^{(i)}_{\delta_{k+1}}(x^{k,i-1})$ with, there exists $v^{k,i}\in\R^{n_i}$ with $\|v^{k,i}\|_{(i)}^*\leq \sqrt{2\delta_{k+1}}$ such that
\[
\frac{\Psi_i}{L_i}(y)-\frac{\Psi_i}{L_i}(x^{k,i}_i)+\delta_{k+1} \geq\ip{v^{k,i}-\frac{1}{L_i}\nabla_i f(x^{k,i-1})-B_iT^{(i)}_{\delta_{k+1}}(x^{k,i-1})}{y-x^{k,i}_i},
\]
for any $y\in\R^{n_i}$. Setting $y=x^*_i$, recognizing that $x^{k,i}_i=x^{k+1}_i$ and $T^{(i)}_{\delta_{k+1}}(x^{k,i-1})=x_i^{k+1}-x_i^k$, and multiplying both sides by $L_i$
\[
\Psi_i(x^*_i)-\Psi_i(x^{k,i}_i)+L_i\delta_{k+1}\geq L_i\ip{v^{k,i}-\frac{1}{L_i}\nabla_i f(x^{k,i-1})+B_i(x_i^k-x_i^{k+1})}{x^*_i-x^{k+1}_i},
\]
Summing this last inequality over $i\in\{1,\ldots,p\}$ we see
\begin{equation*}
\Psi(x^*)-\Psi(x^{k+1})+\sum_{i=1}^p L_i \delta_{k+1}\geq \sum_{i=1}^p L_i\ip{v^{k,i}-\frac{1}{L_i}\nabla_i f(x^{k,i-1})+B_i(x_i^k-x_i^{k+1})}{x^*_i-x^{k+1}_i}
\end{equation*}
\newpage\noindent which we will eventually use in the rearranged form
\begin{multline}\label{eqn:prox:convex:total-nonsmooth}
\sum_{i=1}^p L_i\left(\delta_{k+1}+\ip{v^{k,i}-\frac{1}{L_i}\nabla_i f(x^{k,i-1})+B_i(x_i^k-x_i^{k+1})}{x^{k+1}_i-x^*_i}\right)\\
\geq \Psi(x^{k+1})-\Psi(x^*).
\end{multline}
The convexity of $f$ implies that 
\begin{align*}
F(x^{k+1})-F^*&=f(x^{k+1})-f(x^*)+\Psi(x^{k+1})-\Psi(x^*)\\
&\leq\ip{\nabla f(x^{k+1})}{x^{k+1}-x^*}+\Psi(x^{k+1})-\Psi(x^*)\\
&\leq\sum_{i=1}^p \ip{\nabla_i f(x^{k+1})}{x^{k+1}_i-x^*_i}+\Psi(x^{k+1})-\Psi(x^*)
\end{align*}
which we combine with \eqref{eqn:prox:convex:total-nonsmooth} to yield
\begin{multline}
F(x^{k+1})-F^*\leq\sum_{i=1}^p \ip{\nabla_i f(x^{k+1})}{x^{k+1}_i-x^*_i}\\
+\sum_{i=1}^p L_i\left(\delta_{k+1}+\ip{v^{k,i}-\frac{1}{L_i}\nabla_i f(x^{k,i-1})+B_i(x_i^k-x_i^{k+1})}{x^{k+1}_i-x^*_i}\right)\\
=\sum_{i=1}^p \left(L_i\delta_{k+1}+\ip{\nabla_i f(x^{k+1})-\nabla_i f(x^{k,i-1})+L_iB_i(x^k_i-x^{k+1}_i)+L_i v^{k,i}}{x^{k+1}_i-x^*_i}\right).
\end{multline}

We further compute
\begin{multline}\label{eq:sq-subopt-gap-1}
F(x^{k+1})-F^*\leq\sum_{i=1}^p \bigg[\|\nabla_i f(x^{k+1})-\nabla_i f(x^{k,i-1})\|^*_{(i)}+L_i\|x^k_i-x^{k+1}_i\|_{(i)}\\
+L_i\|v^{k,i}\|_{(i)}^*\bigg]\cdot\|x^{k+1}_i-x^*_i\|_{(i)}+pL_{\max}\delta_{k+1}\\
\leq \sum_{i=1}^p \left[L_f\|x^{k+1}-x^{k,i-1}\|_B+L_{\max}\|x^k_i-x^{k+1}_i\|_{(i)}+L_i\sqrt{2\delta_{k+1}}\right]\cdot\|x^{k+1}_i-x^*_i\|_{(i)}\\
\hspace{18em}+pL_{\max}\delta_{k+1}\\
=(L_f+L_{\max})\|x^k-x^{k+1}\|_B\cdot\sum_{i=1}^p \|x^{k+1}_i-x^*_i\|_{(i)}+L_{\max}\cdot\left(\sum_{i=1}^p \|x^{k+1}_i-x^*_i\|_{(i)}\right)\\
\cdot\sqrt{2\delta_{k+1}}+pL_{\max}\delta_{k+1}
\end{multline}
where we use the Cauchy-Schwarz, triangle, block smoothness \eqref{eq:smooth:block}, and preconditioned smoothness \eqref{eq:smooth} inequalities along with the norm bounds  on the $v^{k,i}$ terms on line 1.
The norm equivalence bound $\|\cdot\|_1\leq p^{1/2}\|\cdot\|_2$ on $\R^p$ along with the coercivity assumption implies
\[
\sum_{i=1}^p \|x^{k+1}_i-x^*_i\|_{(i)}\leq p^{1/2}\sqrt{\sum_{i=1}^p \|x^{k+1}_i-x^*_i\|_{(i)}^2}=p^{1/2}\|x^{k+1}-x^*\|_B\leq p^{1/2}R\left(x^0\right)
\]
so we may refine \eqref{eq:sq-subopt-gap-1} to
\[
F(x^{k+1})-F^*\leq p^{1/2}(L_f+L_{\max})R(x^0)\|x^k-x^{k+1}\|_B+p^{1/2}L_{\max}R(x^0)\sqrt{2\delta_{k+1}}+pL_{\max}\delta_{k+1}
\]
Now, observe that by squaring and applying the Cauchy-Schwarz inequality and monotonicity of the sequence $\{\delta_k\}_{k\geq 1}$, we get

\begin{eqnarray}
[F(x^{k+1})-F^*]^2 \leq \Big[ p^{1/2}(L_f+L_{\max})R(x^0)\|x^k-x^{k+1}\|_B+p^{1/2}L_{\max}R(x^0)\sqrt{2\delta_{k+1}}+pL_{\max}\delta_{k+1} \Big]^2\notag\\
\leq 2p(L_f+L_{\max})^2R(x^0)^2\|x^k-x^{k+1}\|_B^2  + 2\delta_{k+1} (\sqrt{2} p^{1/2}L_{\max}R(x^0) + pL_{\max}\sqrt{\delta_{k+1}}  )^2\label{eq:CS-app-2}\\
\leq 2p(L_f+L_{\max})^2R(x^0)^2\|x^k-x^{k+1}\|_B^2  + 2\delta_{k+1} (\sqrt{2} p^{1/2}L_{\max}R(x^0) + pL_{\max}\sqrt{\delta_1}  )^2\notag
\end{eqnarray}

At this point,  we multiply both sides of the inequality by $\frac{L_{\min}}{8p(L_f+L_{\max})^2 R(x^0)^2}$ and apply the bound of Lemma \ref{lemma:prox:block-decrease} in straightforward fashion to obtain
\begin{multline}
\frac{L_{\min}}{8p(L_f+L_{\max})^2  R(x^0)^2 }[F(x^{k+1})-F^*]^2 \leq \frac{L_{\min}}{4}\| x^k-x^{k+1}\|_B^2  \\
\hspace{12em}+\left( \frac{L_{\min}(\sqrt{2} p^{1/2}L_{\max}R(x^0) + pL_{\max}\sqrt{\delta_1}  )^2}{4p(L_f + L_{\max})^2R(x^0)^2} \right)\delta_{k+1} \\
\leq  F(x^k)-F(x^{k+1}) + 3L_{\min}p\delta_{k+1} +  \frac{L_{\min}L^2_{\max}}{4}\Bigg(   \frac{R(x^0)\sqrt{2} + \sqrt{p\delta_1}}{(L_f + L_{\max})R(x^0)} \Bigg)^2 \delta_{k+1}
\end{multline}
\end{proof}

As promised, we see that stating and proving our main convergence results hinges upon determining the convergence rate of a sequence satisfying a certain recurrence inequality. The following technical lemma, whose proof we defer to this paper's singular appendix, accomplishes this task.

\begin{lemma}\label{lemma:recurrence-technical}
If $\{A_\ell\}_{\ell\geq 0}$ and $\{\Delta_\ell\}_{\ell\geq 1}$ are non-negative, non-increasing sequences of real numbers satisfying the recurrence inequality
\begin{equation}\label{eq:recur-ineq}
\frac{1}{\gamma}A_{\ell+1}^2\leq A_\ell-A_{\ell+1}+\Delta_{\ell+1}.
\end{equation}
for some $\gamma\geq 1$ then the following hold:
\begin{enumerate}[(i)]
\item If $\{\Delta_{\ell}\}_{\ell\geq 1}$ is a constant sequence such that $\Delta_\ell = \Delta\geq 0$ for all $\ell\geq 1$,  then for $u = \sqrt{\Delta\gamma}$, we have that 
\[
A_k\leq\max\left\{\frac{4\gamma (A_0 - u)}{(k-1)(A_0+3u)} +u ,\left(\frac{1}{2}\right)^{(k-1)/2}A_0\right\}
\]
for $k\geq 2$.
	\item If $\{\Delta_\ell\}_{\ell\geq 1}$ shrinks at the sublinear rate $\mathcal{O}(1/k^2)$, i.e. there exists $D>0$ such that $\Delta_\ell\leq D/\ell^2$ for $\ell\geq 1$, then
	\[
A_k\leq \max\left\{\frac{16\gamma}{k-3},\frac{8\sqrt{D\gamma}}{k-3},\left(\frac{1}{2}\right)^{(k-1)/2}A_0\right\}
\]
for $k\geq 4$.

\end{enumerate}
\end{lemma}
%
%

Below we present convergence rates for I-CBPG and thus complete our theoretical developments. Our first two results, Theorem \ref{theorem:convergence-bounded-error} and Corollary \ref{cor:convergence-bounded-error}, cover merely fixed errors. A reader familiar with the analyses of cyclic BPG methods in \cite{Beck13, Beck17} will notice that the constant $\gamma$ in Theorem \ref{theorem:convergence-bounded-error} differs by a factor of $4$ from that in \cite[Theorem 11.18]{Beck17}. This constant, and thus the rate in the exact computation setting, is recoverable from our analysis with minor modification. Namely, by replacing the Cauchy-Schwarz derived bounds in equations \eqref{eq:CS-app-1} and \eqref{eq:CS-app-2}, we can recover said constant. The cost, however, is that the dependence in Lemma \ref{lemma:prox:convex:sufficient-decrease} on $\{\delta_k\}_{k\geq 1}$ deteriorates from $\mathcal{O}(\delta_k)$ to $\mathcal{O}(\delta_k^{1/2})$.

\begin{theorem}[Convergence of I-CBPG: Fixed Error Case]
\label{theorem:convergence-bounded-error}
Let $\{x^k\}_{k\geq 0}$ and $\{\delta_k\}_{k\geq 1}$ denote the sequences of iterates and error tolerances generated by I-CBPG (Algorithm \ref{algo.main}). If the error tolerance sequence is fixed ($\delta_\ell=\delta\geq 0$ for $\ell\geq 1$) then for any $k \ge 2$,
\begin{equation}\label{eqn:convergence-bounded-error}
F(x^k) - F^* \le \max\left\{  \left(\frac{1}{2} \right)^{(k-1)/2} \left( F(x^0)  - F^*  \right ) 		, \;\; \frac{4\gamma \left( F(x^0)  - F^*-u \right ) }{(k-1)(F(x^0)  - F^*  + 3u)} + u	\right\} ,
\end{equation}
where
\[
\gamma= \frac{8p(L_f + L_{\max})^2R(x^0)^2}{L_{\min}} , \quad u  = \sqrt{  L_{\min}\left[3p +    \frac{L^2_{\max}}{4}\left(  \frac{R(x^0)\sqrt{2} + \sqrt{p\delta}}{(L_f + L_{\max})R(x^0)} \right)^2\right]\delta \gamma}
\]
\end{theorem}

\begin{proof}
The result is immediate upon invoking the technical recurrence lemma (Lemma \ref{lemma:recurrence-technical}(i)) with $A_k = F(x^k) - F^*$, $\gamma$ and $u$ as in the statement of the theorem, and 
\[
\Delta =   L_{\min}\left[3p +    \frac{L^2_{\max}}{4}\left(  \frac{R(x^0)\sqrt{2} + \sqrt{p\delta}}{(L_f + L_{\max})R(x^0)} \right)^2\right]\delta 
\]
\end{proof}

\begin{corollary}[Convergence of I-CBPG: Fixed Error Case (Restated)]\label{cor:convergence-bounded-error}
Under the same assumptions and definitions of $u$ and $\gamma$ in Theorem \ref{theorem:convergence-bounded-error}, if $\epsilon > u$ the iterates of I-CBPG (Algorithms \ref{algo.main}) achieve $F(x^k) - F^*\leq\epsilon$ for $k\geq K$ where
\begin{equation}
K = 1+\left\lceil\max \left\{\frac{2}{\log 2}\cdot\log \frac{F(x^0)-F^*}{\epsilon},  \frac{4\gamma(F(x^0)-F^*-u)}{(\epsilon-u)(F(x^0)-F^*+3u)}   \right\} \right\rceil
\end{equation}

\end{corollary}

\begin{proof}
Clearly, the expression for $K$ is the smallest $k\geq 2$ ensuring the right-hand side of \eqref{eqn:convergence-bounded-error} from Theorem \ref{theorem:convergence-bounded-error} is less than or equal $\epsilon$.
\end{proof}

Below we present the convergence rate when the error tolerance sequence $\{\delta_k\}_{k\geq 1}$ decreases at the sublinear rate $\mathcal{O}(1/k^2)$. It is appropriate to reiterate that we are not aware of any works on inexact coordinate descent type methods that establish a rate of decrease on the error sequence that preserves standard convergence rates with exception to the $\ell_1$-norm specialized methods in \cite{Hua12}. Section \ref{sec:numerical}'s numerical experiments strikingly illustrate the benefits of such a decreasing error tolerance sequence and  imply the potential for immense computational savings by permitting higher error tolerances during early iterations.  

\begin{theorem}[Convergence of I-CBPG: Decreasing Error Case]
\label{theorem:convergence-decreasing-error}
Let $\{x^k\}_{k\geq 0}$ and $\{\delta_k\}_{k\geq 1}$ denote the sequences of iterates and error tolerances generated by I-CBPG (Algorithm \ref{algo.main}). If the error tolerance sequence dynamically decreases at the sublinear rate $\mathcal{O}(1/k^2)$ then for any $k \ge 4$, 
\begin{equation}\label{eqn:convergence-decreasing-error}
F(x^k) - F^* \leq \max\left\{\left(\frac{1}{2}\right)^{(k-1)/2}[F(x^0)-F^*],\frac{16\gamma}{k-3},\frac{8\sqrt{D\gamma}}{k-3}\right\}
\end{equation}
where
\[
\gamma = \frac{8p(L_f+L_{\max})^2R(x^0)^2}{L_{\min}} \quad \textrm{ and }\quad D =   \tilde{D}L_{\min}\left[3p +    \frac{L^2_{\max}}{4}\left(  \frac{R(x^0)\sqrt{2} + \sqrt{p\delta_1}}{(L_f + L_{\max})R(x^0)} \right)^2\right]
\]
and $\tilde{D}>0$ is a constant satisfying $\delta_k\leq\frac{\tilde{D}}{k^2}$  for all $k\geq 1$.
\end{theorem}

\begin{proof}
As in the proof of Theorem \ref{theorem:convergence-bounded-error}, the result rests on appropriate identification of the sequence $\{A_\ell\}_{\ell\geq 0}$ and the constants $\gamma$, $D$, and $\lambda$ in Lemma \ref{lemma:recurrence-technical}(ii). The identification here is more straightforward than in the fixed error case. Clearly, a quick examination of \eqref{eq:recur-main} from Lemma \ref{lemma:prox:convex:sufficient-decrease} shows we should choose $\gamma$ as given in the theorem statement, and $\Delta_\ell=\frac{D}{\ell^2}$ to ensure
\[
\frac{1}{\gamma}A_{k+1}^2 \leq A_{k} - A_{k+1} +\frac{D}{k^2}
\]
Invoking Lemma \ref{lemma:recurrence-technical}, then, we achieve
\[
F(x^k) - F^*=A_k\leq\max\left\{\left(\frac{1}{2}\right)^{(k-1)/2}[F(x^0)-F^*],\frac{16\gamma}{k-3},\frac{8\sqrt{D\gamma}}{k-3}\right\}
\]
for $k\geq 4$.
\end{proof}

\begin{corollary}[Convergence of I-CBPG: Decreasing Error Case (Restated)]\label{cor:convergence-decreasing-error}
Under the same assumptions and definitions of $\gamma$ and $D$ in Theorem \ref{theorem:convergence-decreasing-error}, the iterates of I-CBPG (Algorithm \ref{algo.main}) achieve $F(x^k) - F^* \leq\epsilon$ for $k\geq K$  where
\[
K =  \left\lceil \max \left\{1 + \frac{2}{\log 2}\cdot\log\left( \frac{F(x^0)-F^{*}}{\epsilon}\right),3+\frac{16\gamma}{\epsilon},  3 + \frac{8\sqrt{D\gamma}}{\epsilon}\right\} \right\rceil
\]
\end{corollary}

\begin{proof}
Clearly, the expression for $K$ is the smallest $k\geq 2$ ensuring the right-hand side of \eqref{eqn:convergence-decreasing-error} from Theorem \ref{theorem:convergence-decreasing-error} is less than or equal to $\epsilon$.
\end{proof}


%
%


\section{Numerical Experiments}\label{sec:numerical}
In this section we present numerical experiments that demonstrate I-CBPG's performance capabilities.\footnote{Data and related code for these experiments will be made available upon reasonable request.} 
 We selected a common testbed for  such experiments, the LASSO problem,
\begin{equation}\label{num:problem}
\min_{x\in \mathbb{R}^n} \frac{1}{2}\| Ax - b\|_2^2 + \lambda \| x\|_1
\end{equation}
with $A \in \mathbb{R}^{m \times n}$, $b \in \mathbb{R}^m$, $\lambda > 0$. This problem fits within the template of \eqref{eq.problem} by recognizing $f(x) = \frac{1}{2}\| Ax - b\|_2^2 $ and $\Psi_i(U_i^Tx) = \lambda \|  U_i^T x  \|_1$ for $i = 1, \ldots, p$ with $(U_1,\ldots,U_p)=I_n$.  

Throughout, we follow the setup of Section 8.2 of  \cite{Richtarik14} and explore two cases with $N = 10^5$: in the first, the matrix $A$ is ``tall" with size $N \times 0.5N$; in the second,  $A$ is ``wide" with size $N \times 2N$. In both settings, $A$ is a randomly generated sparse matrix with approximately 20 nonzero entries per column. The nonzero entries are generated according to the uniform distribution on $[0,1]$. We subdivide both $A$ matrices into $p = 10$ blocks $A_i$ of equal size, and add to each block an identity matrix padded with zeros to guarantee that each $A_i$ is of full rank. With this, we have that the block smoothness condition \eqref{eq:smooth:block} is satisfied for $B_i  =A_i^TA_i$ and $L_i$ =1. We set $\lambda = 0.1$ when $A$ is $N \times 0.5N$ and $\lambda  = 0.01$ when $\lambda = N \times 2N$. In both cases, we generated $b$ via the uniform distribution on the appropriate unit sphere. 

\begin{figure}[htb!]
\begin{subfigure}{.5\textwidth}
  \centering
  \includegraphics[width=.8\linewidth]{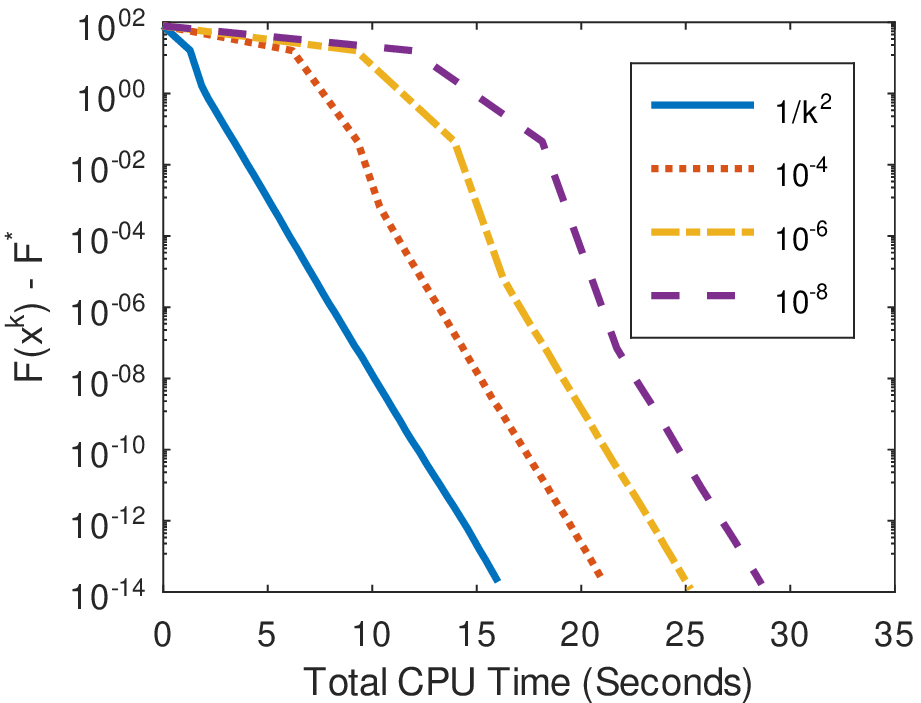}
  \caption{$F(x^k)-F^*$ vs. Total CPU Time, by $\delta_k$ Rule}
  \label{fig:tallmatrix_sub1}
\end{subfigure}%
\begin{subfigure}{.5\textwidth}
  \centering
  \includegraphics[width=.8\linewidth]{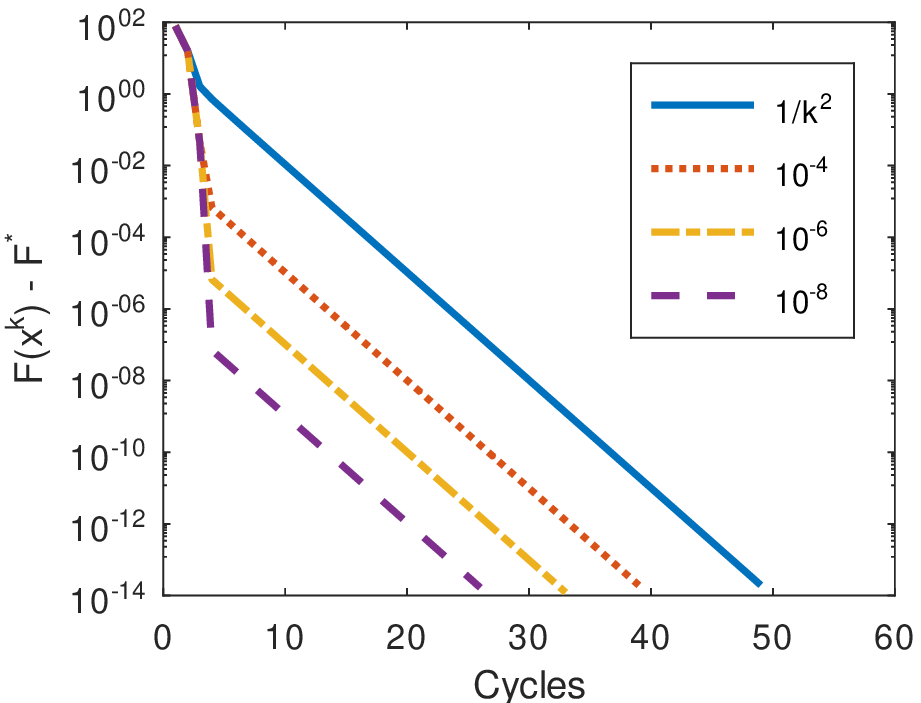}
   \caption{$F(x^k)-F^*$ vs. Cycles, by $\delta_k$ Rule}
  \label{fig:tallmatrix_sub2}
\end{subfigure}%

\caption{ I-CBPG performance graphs for LASSO problem \eqref{num:problem} with ``tall" $A$ (size $N \times 0.5N$, $N = 10^5$)  by $\delta_k$ (error tolerance) rule.}
\label{fig:tallmatrix}
\end{figure}

 \begin{table}[htb!]
  \begin{center}
\scalebox{1}{
  \begin{tabular}{ | c | c | c | c | c | c |  }
  \hline
  \multicolumn{2}{| c | }{Error Tolerance} & $\delta_k = 1/k^2$ & $\delta_k = 10^{-4}$ & $\delta_k = 10^{-6}$ & $\delta_k = 10^{-8}$ \\
  \hline \multicolumn{6}{|c|}{Results at Convergence} 
\\ \hline\multicolumn{2}{|c|}{Total Cycles} & 48  & 38  & 32  & 25 \\  
\multicolumn{2}{|c|}{Total Time}   & 16.0132  & 21.0617  & 25.2328  & 28.6197 \\  
 \hline \multicolumn{6}{|c|}{Results by Cycle Number $k$} \\ \hline01 &  Cycle Time                    & 1.28783  & 6.18858  & 9.23193  & 11.9033 \\  
   & $F(x^k)-F^*$                   & 16.241  & 15.961  & 15.961  & 15.961 \\  
 \hline 05 &  Cycle Time                    & 0.304581  & 0.294988  & 0.295898  & 0.293403 \\  
   & $F(x^k)-F^*$                   & 0.174303  & 0.00016645  & 1.64993e-06  & 1.76299e-08 \\  
 \hline 15 &  Cycle Time                    & 0.286501  & 0.298022  & 0.308174  & 0.299533 \\  
   & $F(x^k)-F^*$                   & 0.00016883  & 1.62136e-07  & 1.61117e-09  & 1.72165e-11 \\  
 \hline 25 &  Cycle Time                    & 0.306331  & 0.32485  & 0.310867  & 0.305136 \\  
   & $F(x^k)-F^*$                   & 1.64872e-07  & 1.58336e-10  & 1.5733e-12  & 1.66533e-14 \\  
 \hline 35 &  Cycle Time                    & 0.321347  & 0.315409  &    &   \\  
   & $F(x^k)-F^*$                   & 1.61007e-10  & 1.54543e-13  &    &   \\  
 \hline 45 &  Cycle Time                    & 0.288697  &    &    &   \\  
   & $F(x^k)-F^*$                   & 1.57097e-13  &    &    &   \\  
 \hline 
   \hline
 \end{tabular}}
  \end{center}
  \caption{Total cycles, total elapsed CPU time, CPU time per cycle, and suboptimality gap by cycle as a function of $\delta_k$  (error tolerance) rule for ``tall" $A$ (size $N \times 0.5N$, $N = 10^5$).}
  \label{tab:tallmatrix} 
\end{table}

At each step of the algorithm, to compute our update to the $i^{th}$ block, we find $T_{\delta_{k+1}}^{\left( i\right)}\left( x^{k,i-1} \right) $ by calculating a $\delta_{k+1}$-approximate solution (that also satisfies the monotonic decrease condition \eqref{eq.mondec}) to the smaller-dimensional (likely far smaller) problem
\begin{equation}\label{num:subproblasso}
\arg \min_{y \in \mathbb{R}^{n_i}}  \frac{1}{2} \| A_i y - \tilde{b}^{k,i}\|_2^2 + \lambda \| y\|_1
\end{equation}
where $\tilde{b}^{k,i} := b-Ax^{k,i-1}+A_ix^{k,i-1}_i$. To do so, we use the box-constrained gradient projection algorithm of \cite{Broughton11} to approximately solve \eqref{num:subproblasso}. We terminate the box-constrained gradient projection algorithm when the duality gap for \eqref{num:subproblasso} descends below $\delta_{k+1}$ and the monotonic decrease condition is satisfied.

We explore the effect of different error tolerance levels on I-CBPG runtime performance and cycle counts by comparing algorithm behavior under  a dynamic rule that sets $\delta_k = 1/k^2$ with three different constant rules that fix $\delta_k$  at $10^{-4}$. $10^{-6}$, or $10^{-8}$ for all cycles, respectively. Plots of the difference $F\left(x^k\right) - F^*$ against both CPU time and the number of iterations $k$ are presented in Figure \ref{fig:tallmatrix} and Table \ref{tab:tallmatrix} for the first case, where $A$ is $N \times 0.5N$, and in Figure \ref{fig:widematrix} and Table \ref{tab:widematrix} for the second case, where $A$ is $N \times 2N$. We performed a number of simulations of both types, but limit our discussion to a single instance of each for definiteness. Owing to the random data generation procedure, we did observe some minor variation in the quantitative values of the various ratios discussed below across different  problem instances, but the qualitative relationships between I-CBPG performance under different error tolerance sequences have been stable across all cases we examined. 

Among the fixed error tolerance regimes ($\delta_k$ constant), it is apparent for both types of $A$ matrix that larger error tolerances (as measured by a higher value of $\delta_k$) translate to achieving a given level of accuracy in a shorter interval of CPU time than is possible with a more stringent (lower) value of $\delta_k$, even though the algorithm runs through more cycles in total to achieve a given suboptimality gap when the error is larger. The property of larger fixed values of $\delta_k$ translating to shorter CPU time costs echoes the results of \cite{Richtarik16}, while our results differ by showing a greater spread across error tolerance regimes in the number of cycles needed to achieve convergence. Examining the performance of the  dynamic rule $\delta_k = 1/k^2$, we see a continuation of this trend for both cases, where the dynamic rule uses the least total CPU time of all, but needs the most cycles to achieve convergence. 

Further study of Figure \ref{fig:tallmatrix} and Table \ref{tab:tallmatrix} explains this apparent discrepancy while showcasing the power and performance advantage of inexact computation. While cycle times display very little variation across different $\delta$ regimes from cycle $k = 5$ onward, a greater degree of error tolerance in early cycles translates to marked improvements in speed. In particular, for the first LASSO example, one sees that more stringent error tolerances come at significantly higher CPU time costs for early iterates. Conversely, more permissive error tolerance rules for early iterates achieve the same progress in a fraction of the time.  These time savings carry through to convergence, as shown in Table \ref{tab:tallmatrix}. In particular, the dynamic error tolerance regime $\delta_k = 1/k^2$ achieves converges in $24\%$ less time than the constant rule $\delta_k = 10^{-4}$, which in turn achieves convergence in $27\%$ less time than the constant rule $\delta_k = 10^{-8}$, so that the dynamic rule $\delta_k = 1/k^2$ is sufficiently fast relative to  the constant rule $\delta_k = 10^{-8}$ that it achieves convergence in $45\%$ less time.

The same trend characterizes the second case, as shown in Figure \ref{fig:widematrix} and Table \ref{tab:widematrix}, with some noteworthy differences. Specifically, total solution time is markedly longer for all $\delta_k $ rules, but the CPU time savings achieved by applying I-CBPG  with relatively more permissive $\delta_k$ rules are even more pronounced. In particular, from Table \ref{tab:widematrix}, we calculate a total CPU time savings of slightly more than 50\% for the $\delta_k = 1/k^2$ rule relative to the $\delta_k = 10^{-4}$ rule. Similarly, we see an even greater total CPU time savings of over 56\% for the more lenient $\delta_k = 10^{-4}$ rule relative to the stricter $\delta_k = 10^{-8}$ rule. Most dramatic of all, we find  a CPU time savings of more than 79\% for the $\delta_k = 1/k^2$ rule relative to the constant rule $\delta_k  =10^{-8}$. Further comparison of Table \ref{tab:widematrix} with Table \ref{fig:tallmatrix} helps explain the additional performance benefit offered by inexact computation in this case by showing that nontrivial differences in cycle times across error tolerances persist for longer than in the first case, even after the very large differences seen in early cycles have moderated. 

\begin{figure}[htb!]
\begin{subfigure}{.5\textwidth}
  \centering
  \includegraphics[width=0.8\linewidth]{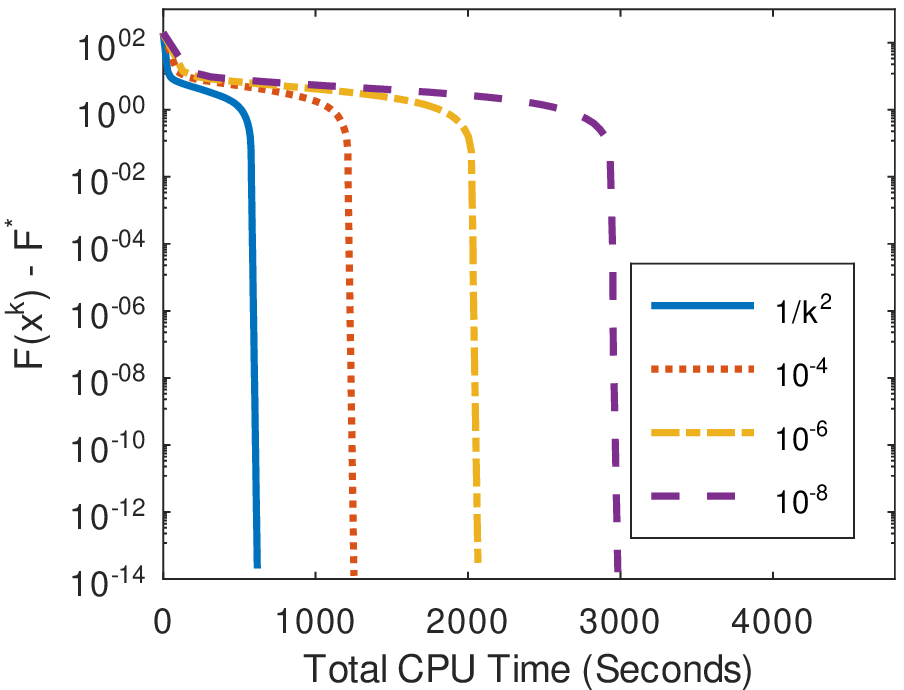}
   \caption{$F(x^k)-F^*$ vs. Total CPU Time, by $\delta_k$ Rule}
  \label{fig:widematrix_sub1}
\end{subfigure}%
\begin{subfigure}{.5\textwidth}
  \centering
  \includegraphics[width=0.8\linewidth]{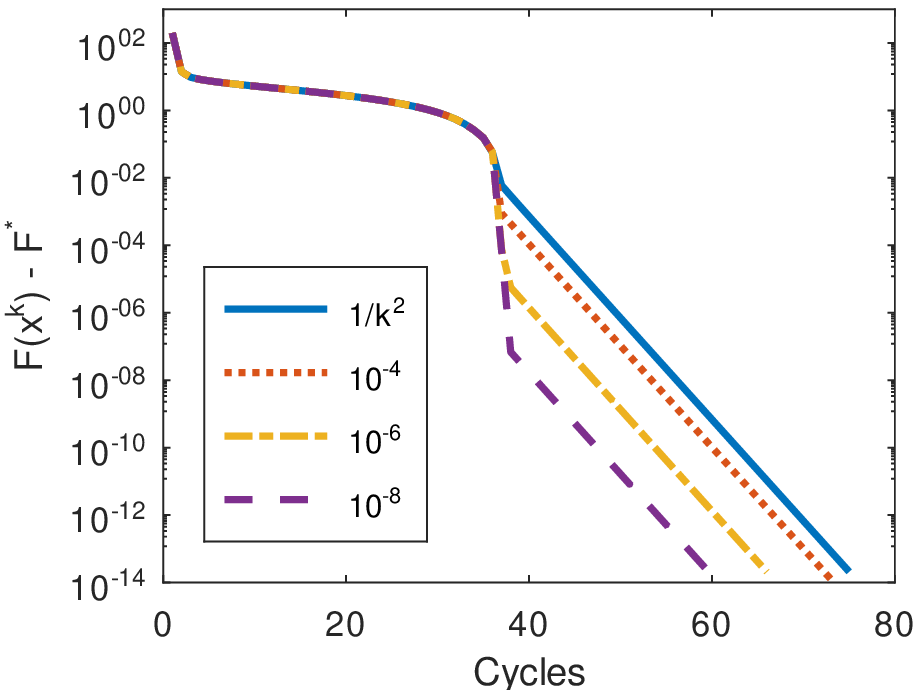}
   \caption{$F(x^k)-F^*$ vs Cycles, by $\delta_k$ Rule}
  \label{fig:widematrix_sub2}
\end{subfigure}%

\caption{ I-CBPG performance graphs for LASSO problem \eqref{num:problem} with ``wide" $A$ (size $N \times 2N$, $N = 10^5$) by $\delta_k$ (error tolerance) rule.}
\label{fig:widematrix}
\end{figure}

 \begin{table}[htb!]
  \begin{center}
\scalebox{1}{
  \begin{tabular}{ | c | c | c | c | c | c |  }
  \hline
  \multicolumn{2}{| c | }{Error Tolerance} & $\delta_k = 1/k^2$ & $\delta_k = 10^{-4}$ & $\delta_k = 10^{-6}$ & $\delta_k = 10^{-8}$ \\
  \hline \multicolumn{6}{|c|}{Results at Convergence} 
\\ \hline\multicolumn{2}{|c|}{Total Cycles} & 74  & 72  & 65  & 59 \\  
\multicolumn{2}{|c|}{Total Time}   & 617.614  & 1251.38  & 2065.55  & 2982.54 \\  
 \hline \multicolumn{6}{|c|}{Results by Cycle Number $k$} \\ \hline01 &  Cycle Time                    & 29.6646  & 89.9671  & 129.238  & 164.616 \\  
   & $F(x^k)-F^*$                   & 14.0733  & 14.0743  & 14.0743  & 14.0743 \\  
 \hline 05 &  Cycle Time                    & 17.2434  & 51.2102  & 79.3759  & 99.8509 \\  
   & $F(x^k)-F^*$                   & 6.9191  & 6.91677  & 6.91677  & 6.91677 \\  
 \hline 15 &  Cycle Time                    & 19.5706  & 34.3488  & 61.0178  & 92.177 \\  
   & $F(x^k)-F^*$                   & 3.62617  & 3.62753  & 3.62753  & 3.62753 \\  
 \hline 25 &  Cycle Time                    & 16.5504  & 25.798  & 47.0633  & 67.9062 \\  
   & $F(x^k)-F^*$                   & 1.57601  & 1.57736  & 1.57737  & 1.57737 \\  
 \hline 35 &  Cycle Time                    & 5.38051  & 8.56292  & 18.2444  & 28.9568 \\  
   & $F(x^k)-F^*$                   & 0.0565171  & 0.0585538  & 0.0586275  & 0.0586271 \\  
 \hline 45 &  Cycle Time                    & 0.999437  & 0.970537  & 1.0669  & 1.30292 \\  
   & $F(x^k)-F^*$                   & 1.12488e-05  & 1.68518e-06  & 2.17146e-08  & 2.6541e-10 \\  
 \hline 55 &  Cycle Time                    & 1.04114  & 0.962187  & 1.1745  & 1.23398 \\  
   & $F(x^k)-F^*$                   & 1.09842e-08  & 1.64561e-09  & 2.12048e-11  & 2.59183e-13 \\  
 \hline 65 &  Cycle Time                    & 0.949509  & 0.965906  & 1.14957  &   \\  
   & $F(x^k)-F^*$                   & 1.07268e-11  & 1.60704e-12  & 2.07083e-14  &   \\  
 \hline 74 &  Cycle Time                    & 1.01878  &    &    &   \\  
   & $F(x^k)-F^*$                   & 2.09511e-14  &    &    &   \\  
 \hline 
   \hline
 \end{tabular}}
  \end{center}
  \caption{Total cycles, total elapsed CPU time, CPU time per cycle, and suboptimality gap by cycle as a function of  $\delta_k$  (error tolerance) rule for ``wide" $A$ (size $N \times 2N$, $N = 10^5$).}
  \label{tab:widematrix} 
\end{table} 

\newpage

\section{Conclusion}

In this paper, we introduced inexactly computed gradients and proximal maps into the Cyclic Block Proximal Gradient scheme resulting in our I-CBPG algorithm. Our convergence analysis covers both dynamically decreasing and fixed error tolerances. Our numerical experiments imply that dynamically decreasing error tolerances may greatly reduce the CPU time cost of I-CBPG's early iterations. In the course of our analysis, we explored a unified framework for analyzing inexact computation via inexactly computed pre-conditioned proximal maps. This framework's tools enabled us to show how the inexact proximal map subsumes the inexact computation of both gradients and proximal maps as well as a notion of inexact computation provided in \cite{Rockafellar76}.

\clearpage

\bibliographystyle{plain}
\bibliography{mybibnew}

\begin{thebibliography}{10}

\bibitem{Beck17}
Amir Beck.
\newblock {\em First-order methods in optimization}.
\newblock SIAM, 2017.

\bibitem{Beck13}
Amir Beck and Luba Tetruashvili.
\newblock On the convergence of block coordinate descent type methods.
\newblock {\em SIAM Journal on Optimization}, 23(4):2037--2060, 2013.

\bibitem{Broughton11}
Robert~L Broughton, Ian~D Coope, Peter~F Renaud, and REH Tappenden.
\newblock A box constrained gradient projection algorithm for compressed
  sensing.
\newblock {\em Signal processing}, 91(8):1985--1992, 2011.

\bibitem{Alexandre_2008}
Alexandre D'Aspremont.
\newblock Smooth optimization with approximate gradient.
\newblock {\em SIAM Journal on Optimization}, 19(3):1171--1183, 2008.

\bibitem{Devolder14}
Olivier Devolder, Fran{\c{c}}ois Glineur, and Yurii Nesterov.
\newblock First-order methods of smooth convex optimization with inexact
  oracle.
\newblock {\em Mathematical Programming}, 146(1):37--75, 2014.

\bibitem{Devolder13}
Olivier Devolder, Fran{\c{c}}ois Glineur, Yurii Nesterov, et~al.
\newblock Intermediate gradient methods for smooth convex problems with inexact
  oracle.
\newblock Technical report, Technical report, CORE-2013017, 2013.

\bibitem{Donoho06}
David~L Donoho.
\newblock Compressed sensing.
\newblock {\em IEEE Transactions on information theory}, 52(4):1289--1306,
  2006.

\bibitem{Dvurechensky16}
Pavel Dvurechensky and Alexander Gasnikov.
\newblock Stochastic intermediate gradient method for convex problems with
  stochastic inexact oracle.
\newblock {\em Journal of Optimization Theory and Applications},
  171(1):121--145, 2016.

\bibitem{FrongilloReid15}
Rafael Frongillo and Mark~D Reid.
\newblock Convergence analysis of prediction markets via randomized subspace
  descent.
\newblock In C.~Cortes, N.~D. Lawrence, D.~D. Lee, M.~Sugiyama, and R.~Garnett,
  editors, {\em Advances in Neural Information Processing Systems 28}, pages
  3034--3042. Curran Associates, Inc., 2015.

\bibitem{Hiriart13}
Jean-Baptiste Hiriart-Urruty and Claude Lemar{\'e}chal.
\newblock {\em Convex analysis and minimization algorithms II: Advanced Theory
  and Bundle Methods}.
\newblock Springer science \& business media, 2013.

\bibitem{Hua12}
Xiaoqin Hua and Nobuo Yamashita.
\newblock An inexact coordinate descent method for the weighted l1-regularized
  convex optimization problem.
\newblock {\em Pacific Journal of Optimization}, 9(4), 2013.

\bibitem{Leventhal10}
Dennis Leventhal and Adrian~S Lewis.
\newblock Randomized methods for linear constraints: convergence rates and
  conditioning.
\newblock {\em Mathematics of Operations Research}, 35(3):641--654, 2010.

\bibitem{Lu17}
Haihao Lu, Robert~M Freund, and Yurii Nesterov.
\newblock Relatively smooth convex optimization by first-order methods, and
  applications.
\newblock {\em SIAM Journal on Optimization}, 28(1):333--354, 2018.

\bibitem{Nesterov12}
Yurii Nesterov.
\newblock Efficiency of coordinate descent methods on huge-scale optimization
  problems.
\newblock {\em SIAM Journal on Optimization}, 22(2):341--362, 2012.

\bibitem{Qin13}
Zhiwei Qin, Katya Scheinberg, and Donald Goldfarb.
\newblock Efficient block-coordinate descent algorithms for the group lasso.
\newblock {\em Mathematical Programming Computation}, 5(2):143--169, 2013.

\bibitem{Richtarik12}
Peter Richt{\'a}rik and Martin Tak{\'a}{\v{c}}.
\newblock Efficient serial and parallel coordinate descent methods for
  huge-scale truss topology design.
\newblock In {\em Operations Research Proceedings 2011}, pages 27--32.
  Springer, 2012.

\bibitem{Richtarik16}
Peter Richt{\'a}rik and Martin Tak{\'a}{\v{c}}.
\newblock Parallel coordinate descent methods for big data optimization.
\newblock {\em Mathematical Programming}, 156(1-2):433--484, 2016.

\bibitem{Rockafellar76}
R~Tyrrell Rockafellar.
\newblock Monotone operators and the proximal point algorithm.
\newblock {\em SIAM journal on control and optimization}, 14(5):877--898, 1976.

\bibitem{Scardapane17}
Simone Scardapane, Danilo Comminiello, Amir Hussain, and Aurelio Uncini.
\newblock Group sparse regularization for deep neural networks.
\newblock {\em Neurocomputing}, 241:81--89, 2017.

\bibitem{Schmidt11}
Mark Schmidt, Nicolas Roux, and Francis Bach.
\newblock Convergence rates of inexact proximal-gradient methods for convex
  optimization.
\newblock In J.~Shawe-Taylor, R.~Zemel, P.~Bartlett, F.~Pereira, and K.~Q.
  Weinberger, editors, {\em Advances in Neural Information Processing Systems},
  volume~24. Curran Associates, Inc., 2011.

\bibitem{Shefi16}
Ron Shefi and Marc Teboulle.
\newblock On the rate of convergence of the proximal alternating linearized
  minimization algorithm for convex problems.
\newblock {\em EURO Journal on Computational Optimization}, 4(1):27--46, 2016.

\bibitem{Simon12}
Noah Simon and Robert Tibshirani.
\newblock Standardization and the group lasso penalty.
\newblock {\em Statistica Sinica}, 22(2):983--1002, 2012.

\bibitem{Richtarik14}
Rachael Tappenden, Peter Richt{\'a}rik, and Jacek Gondzio.
\newblock Inexact coordinate descent: complexity and preconditioning.
\newblock {\em Journal of Optimization Theory and Applications},
  170(1):144--176, 2016.

\bibitem{Villa13}
Silvia Villa, Saverio Salzo, Luca Baldassarre, and Alessandro Verri.
\newblock Accelerated and inexact forward-backward algorithms.
\newblock {\em SIAM Journal on Optimization}, 23(3):1607--1633, 2013.

\bibitem{Wright09}
Stephen~J Wright, Robert~D Nowak, and M{\'a}rio~AT Figueiredo.
\newblock Sparse reconstruction by separable approximation.
\newblock {\em IEEE Transactions on signal processing}, 57(7):2479--2493, 2009.

\end{thebibliography}

\appendix
\section{Proof of Lemma \ref{lemma:recurrence-technical}}

\begin{proof} Fix $k\geq  2$. We begin by dividing both sides of \eqref{eq:recur-ineq} by $A_\ell A_{\ell+1}$,
\[
\frac{1}{\gamma}\frac{A_{\ell+1}}{A_\ell}\leq\frac{1}{A_{\ell+1}}-\frac{1}{A_\ell}+\frac{\Delta_{\ell+1}}{A_\ell A_{\ell+1}},
\]
rearranging and using monotonicity of $\{A_\ell\}_{\ell\geq 0}$, simplify to
\[
\frac{1}{A_{\ell+1}}-\frac{1}{A_\ell}\geq\frac{1}{\gamma}\frac{A_{\ell+1}}{A_\ell}-\frac{\Delta_{\ell+1}}{A_\ell A_{\ell+1}}\geq\frac{1}{\gamma}\frac{A_{\ell+1}}{A_\ell}+-\frac{\Delta_{\ell+1}}{A_\ell A_{\ell+1}}.
\]
This rearrangement foreshadows the important roles of $A_{\ell+1}/A_\ell$ and $\Delta_{\ell+1}/(A_\ell A_{\ell+1})$. We consider two cases, divided according to the typical size of the ratio  $A_{\ell+1}/A_\ell$ for $\ell + 1 \le k$. In the second case, when the values of $A_\ell$ fall at what one may consider a relatively slow rate over this range, we consider three subcases based on the behavior of $\{\Delta_\ell\}_{\ell\geq1}$ and the typical values of $\frac{\Delta_{\ell+1}}{A_\ell A_{\ell+1}}$.
\begin{enumerate}[(i)]
\item For at least $\lfloor k/2\rfloor$ values of $0\leq \ell\leq k-1$, we have $A_{\ell+1}/A_\ell\leq 1/2$.
\item For at least $\lfloor k/2\rfloor$ values of $0\leq \ell\leq k-1$, we have $1/2  < A_{\ell+1}/A_\ell \leq1$.
In this case, we consider three subcases based on the values of $\Delta_{\ell+1}/(A_\ell A_{\ell+1})$ and the sequence $\{\Delta_\ell\}_{\ell\geq 1}$.
\end{enumerate}

\noindent\textit{Case 1: For at least $\lfloor k/2\rfloor$ values of $0\leq \ell\leq k-1$, $\frac{A_{\ell+1}}{A_\ell}\leq\frac{1}{2}$.}\\

\noindent This is the easy case. First, assume that $k$ is even. Then we have that $A_{\ell+1}\leq\frac{1}{2}A_\ell$ for at least $k/2$ values of $0\leq \ell\leq k-1$ so 
\[
A_k\leq  \left(\frac{1}{2}\right)^{k/2}A_0,
\]
since the $A_\ell$ terms are decreasing. If $k>2$ is odd, then $k-1$ is even, so by the same logic 
\[
A_k\leq  \left(\frac{1}{2}\right)^{(k-1)/2}A_0.
\]

\noindent\textit{Case 2: For at least $\lfloor k/2\rfloor$ values of $0\leq \ell\leq k-1$, $\frac{1}{2} <  \frac{A_{\ell+1}}{A_\ell} \leq 1$.}\\

\noindent We examine the following three subcases in turn:
\begin{enumerate}[(i)]
	\item $\Delta_{\ell} = \Delta\geq 0 $ for all $\ell$.
	\item The sequence $\{\Delta_\ell\}_{\ell\geq1}$ shrinks at the sublinear rate $\mathcal{O}(1/\ell^2)$ and for at least $\lfloor k/4\rfloor$ of the values for which $\frac{1}{2}< \frac{A_{\ell+1}}{A_\ell}\leq 1$ it also holds that $\frac{1}{4\gamma}>\frac{\Delta_{\ell+1}}{A_\ell A_{\ell+1}}$.
	\item The sequence $\{\Delta_\ell\}_{\ell\geq1}$ shrinks at the sublinear rate $\mathcal{O}(1/\ell^2)$ and for at least $\lfloor k/4\rfloor$ of the values for which $\frac{1}{2}<\frac{A_{\ell+1}}{A_\ell}\leq 1$ it also holds that $\frac{1}{4\gamma} \leq \frac{\Delta_{\ell+1}}{A_\ell A_{\ell+1}}$.
\end{enumerate}

\noindent\textit{Case 2, Subcase i: $\Delta_\ell = \Delta\geq 0$ for all $\ell$}.\\

Assume for now that $k$ is even. Define $u=\sqrt{\Delta\gamma}$, and let $\tilde{A}_{\ell}=A_{\ell}-u$.  Then the recurrence   \eqref{eq:recur-ineq} implies that $ \frac{1}{\gamma}A_{\ell+1}^2\leq A_\ell-A_{\ell+1}+\Delta_{\ell+1} $, which we may express as 
\[
\frac{1}{\gamma}(\tilde{A}_{\ell+1}+u)^2=\frac{1}{\gamma}A_{\ell+1}^2\leq  A_\ell-A_{\ell+1}+\Delta=\tilde{A}_\ell-\tilde{A}_{\ell+1}+\Delta
\]
Expanding the square on the left, using the definition of $u$, and rearranging
\[
\frac{1}{\gamma}\tilde{A}_{\ell+1}^2\leq \tilde{A}_\ell-\left(1+\frac{2u}{\gamma}\right)\tilde{A}_{\ell+1}
\]
If $\tilde{A}_{k} \leq 0$, the result is immediate, so suppose that $\tilde{A}_{k} > 0$, from which it follows that earlier terms $\tilde{A}_{\ell}$ are also positive.  Then, for any $\ell$ with $0 \le \ell \le k-1$, we may divide the recurrence inequality by the product $\tilde{A}_{\ell+1}\tilde{A}_\ell$ to obtain
\[
 \frac{1}{\tilde{A}_{\ell+1}} - \left( 1 + \frac{2u}{\gamma}\right)\frac{1}{\tilde{A}_\ell} \ge \frac{1}{\gamma}\frac{\tilde{A}_{\ell+1}}{\tilde{A}_\ell} 
\]
Now, by hypothesis, for at least $k/2$ indices in the range $0 \leq \ell \leq k -1$
\[
 \frac{1}{\tilde{A}_{\ell+1}} - \frac{1}{\tilde{A}_\ell} \geq \frac{1}{\gamma}\frac{\tilde{A}_{\ell+1}}{\tilde{A}_\ell}  +  \frac{2u}{\gamma}  \frac{1}{\tilde{A}_\ell}  \geq \frac{1}{\gamma}\frac{1 }{2}  +  \frac{2u}{\gamma}  \frac{1}{\tilde{A}_0} 
\]
Iterating backward, one obtains 
\[
 \frac{1}{\tilde{A}_k} \geq \frac{1}{\tilde{A}_{k}} - \frac{1}{\tilde{A}_0}  \geq \frac{k}{2} \left( \frac{1}{2\gamma}  +  \frac{2u}{\gamma}  \frac{1}{\tilde{A}_0} \right)
\]
 which gives $\tilde{A}_k \leq 4\gamma \tilde{A}_0/[k (\tilde{A}_0+4u)]$. The result follows from noting that $k-1$ is even if $k$ is odd, so we may replace $k$ with $k-1$ above to obtain a generic bound. \\



\noindent\textit{Case 2, Subcase ii: The sequence $\{\Delta_\ell\}_{\ell\geq1}$ shrinks at the sublinear rate $\mathcal{O}(1/\ell^2)$ and for at least $\lfloor k/4\rfloor$ of the values for which  $\frac{1}{2} <  \frac{A_{\ell+1}}{A_\ell} \leq 1$, it also holds that$\frac{\Delta_{\ell+1}}{A_\ell A_{\ell+1}}<  \frac{1}{4\gamma}$.}\\

Our reasoning follows the same idea as when $\Delta_\ell=\Delta\geq 0$ for all $\ell\geq 1$ (Case 2, Subcase i). First, assume that $k$ is divisible by $4$. We have for $k/4$ values of $0\leq\ell\leq k-1$ that 
\[
\frac{1}{A_{\ell+1}}-\frac{1}{A_\ell}\geq\frac{1}{2\gamma}-\frac{\Delta_{\ell+1}}{A_\ell A_{\ell+1}}\geq\frac{1}{2\gamma}-\frac{1}{4\gamma}=\frac{1}{4\gamma}
\]
This inequality iterated backward, plus monotonicity and non-negativity of the sequence $\{A_\ell\}_{\ell\geq 0}$, implies that
\[
\frac{1}{A_k}\geq\frac{1}{A_k}-\frac{1}{A_0}\geq \frac{k}{4}\left[\frac{1}{4\gamma}\right]=\frac{k}{16\gamma }
\]
Rearranging, we have that $A_k \leq 16\gamma / k$. If $k > 4$ is not divisible by $4$, then $k-1$, $k-2$, or $k-3$ must be, so in the worst case $A_k \leq 16\gamma/(k-3)$.\\ 

\noindent\textit{Case 2, Subcase iii: The sequence $\{\Delta_\ell\}_{\ell\geq1}$ shrinks at the sublinear rate $\mathcal{O}(1/\ell^2)$ and for at least $\lfloor k/4\rfloor$ of the values for which $\frac{1}{2} <  \frac{A_{\ell+1}}{A_\ell} \leq 1$, it also holds that $\frac{\Delta_{\ell+1}}{A_\ell A_{\ell+1}}\geq\frac{1}{4\gamma}$.}\\

First, suppose $k$ is divisible by $4$. Let $\ell^*$ denote the largest  $\ell\in\{0,\ldots,k-1\}$ for which $\frac{\Delta_{\ell+1}}{A_\ell A_{\ell+1}}\geq\frac{1}{4\gamma}$ holds. By hypothesis, $\ell^*$ must be at least as big as $\frac{k}{4} - 1$, and $\Delta^2_\ell \leq D/\ell^2$, so 
\[
\frac{1}{4\gamma}\cdot A_k^2\leq \frac{1}{4\gamma}\cdot A_k A_{k-1}\leq \frac{1}{4\gamma}\cdot A_{\ell^*+1} A_{\ell^*}\leq \Delta_{\ell^*+1}\leq\Delta_{k/4}\leq\frac{D}{(k/4)^2}, 
\]
Dividing by $1/4\gamma$ and taking square roots, we have $A_{k}\leq\frac{8\sqrt{\gamma D}}{k}$. If $k>4$ is not divisible by $4$, then one of $k-1$, $k-2$, or $k-3$ are, so at worst $A_k\leq  \frac{8\sqrt{\gamma D}}{k-3}$.\\

Having completed our analysis, we may now combine the results from Case 1 with the appropriate Subcase(s) of Case 2 to establish the results in Lemma \ref{lemma:recurrence-technical}.

\end{proof}
\end{document}